\newcommand{\wvec}[1]{W\hspace{-0.035 in}\left( #1\right)}
\documentclass[letterpaper]{llncs}

\usepackage{xypic}
\input xy
\xyoption{all}
\usepackage{graphpap}
\usepackage{url}

\include{biblio}

\usepackage{ellnet}

\spnewtheorem*{nonumtheorem}{Theorem}{\normalfont\bfseries}{\itshape}

\title{The elliptic curve discrete logarithm problem and equivalent hard problems for elliptic divisibility sequences}  
\titlerunning{ECDLP and equivalent hard problems for EDS}
\author{Kristin E. Lauter\inst{1} and Katherine E. Stange\inst{2} \fnmsep \thanks{The second author was supported by NSERC Award PGS D2 331379-2006, and this work was performed during an internship of the second author at Microsoft Research.}}      
\date{\today}  

\institute{Microsoft Research, One Microsoft Way, Redmond, WA 98052 \email{klauter@microsoft.com} \and Department of Mathematics, Harvard University, Cambridge, MA 02138 \email{stange@math.harvard.edu}}

\begin{document}

\maketitle 

\begin{abstract}
We define three hard problems in the theory of elliptic divisibility sequences (\emph{EDS Association}, \emph{EDS Residue} and \emph{EDS Discrete Log}), each of which is solvable in sub-exponential time if and only if the elliptic curve discrete logarithm problem is solvable in sub-exponential time.  We also relate the problem of EDS Association to the Tate pairing and the MOV, Frey-R\"{u}ck and Shipsey EDS attacks on the elliptic curve discrete logarithm problem in the cases where these apply. 
\end{abstract}


\section{Introduction}

The security of elliptic curve cryptography rests on the assumption that the \emph{elliptic curve discrete logarithm problem} is hard.

\begin{problem}[Elliptic Curve Discrete Logarithm Problem (ECDLP)]
\label{prob: ecdlp}
Let $E$ be an elliptic curve over a finite field $K$.  Suppose there are points $P, Q \in E(K)$ given such that $Q \in \left< P \right>$.  Determine $k$ such that $Q = [k]P$.
\end{problem}

In this article, we explore several related hard problems with a view to expanding the theoretical foundations of the security of ECDLP as a hard problem.  Our research is inspired by work of Rachel Shipsey in her thesis \cite{Shi}, relating the ECDLP to elliptic divisibility sequences (EDS).  An elliptic divisibility sequence is a recurrence sequence $W(n)$ satisfying the relation
\[
W(n+m)W(n-m)= W(n+1)W(n-1)W(m)^2 -W(m+1)W(m-1)W(n)^2.
\]
We relate Shipsey's work to the MOV and Frey-R\"{u}ck attacks and explain their limitations from the EDS point of view.  We also point to a specific avenue for attacking ECDLP by analysing the quadratic residuosity of elliptic divisibility sequences.

The study of elliptic divisibility sequences was introduced by Morgan Ward \cite{War}.   Let $\Psi_n$ denote the $n$-th division polynomial of an elliptic curve $E$ over the rationals.  The sequence $W_{E,P}: \Z \rightarrow \Q$ of the form $W_{E,P}(n) = \Psi_n(P)$ for some fixed point $P \in E(\Q)$ is an elliptic divisibility sequence, and Ward showed that almost all elliptic divisibility sequences arise in this way.  This relationship is the basis of our work here.

The general theory has been developed by Swart \cite{Swa}, Ayad \cite{Aya}, Silverman \cite{Sil5}\cite{Sil4}, Everest, McLaren and Thomas Ward \cite{EveMclWar} and, more recently, generalised to higher rank \emph{elliptic nets} by Stange \cite{Sta4}\cite{Sta1}.  For an overview of research, see \cite{EvePooShpWar}.  Sections \ref{sec: ellnets} and \ref{sec: perfper} provide brief background on elliptic divisibility sequences and elliptic nets, more information about which can be found in \cite{Sta4}\cite{Sta1}\cite{Sta3}.

The hard problems for elliptic divisibility sequences we consider are:

\begin{problem}[EDS Association]
\label{prob: seqass}
Let $E$ be an elliptic curve over a finite field $K$.  Suppose there are points $P, Q \in E(K)$ given such that $Q \in \left< P \right>$, $Q \neq \mathcal{O}$, and $\operatorname{ord}(P) \geq 4$.  Determine $W_{E,P}(k)$ for $0 < k < \operatorname{ord}(P)$ such that $Q = [k]P$.
\end{problem}

\begin{problem}[EDS Residue]
\label{prob: seqres}
Let $E$ be an elliptic curve over a finite field $K$.  Suppose there are points $P, Q \in E(K)$ given such that $Q \in \left< P \right>$, $Q \neq \mathcal{O}$, and $\operatorname{ord}(P) \geq 4$.  Determine the quadratic residuosity of $W_{E,P}(k)$ for $0 < k < \operatorname{ord}(P)$ such that $Q = [k]P$.
\end{problem}

\begin{problem}[Width s EDS Discrete Log]
\label{prob: edsdlp}
Given an elliptic divisibility sequence $W$ and terms $W(k)$, $W(k+1)$, $\ldots$, $W(k+s-1)$, determine $k$.
\end{problem}

Problem \ref{prob: edsdlp} was considered by Shipsey \cite[\S 6.3.1]{Shi} and Gosper, Orman and Schroeppel \cite[\S 3]{GosOrmSch}.  Problem \ref{prob: seqass} is also implicit in \cite[\S 6.4.1]{Shi} and \cite[\S 3]{GosOrmSch}.

A perfectly periodic elliptic divisibility sequence is one which has a finite period $n$ and whose first positive index $k$ at which $W(k) = 0$ is $k=n$.  If a sequence is not perfectly periodic, then it has period $n > k$.  In Section \ref{sec: equiv}, we prove the following theorem.

\begin{theorem} \label{mainthm}
Let $E$ be an elliptic curve over a finite field $K=\F_q$ of characteristic $\neq 2$.  If any one of the following problems is solvable in sub-exponential time, then all of them are:
\begin{enumerate}
\item Problem \ref{prob: ecdlp}:  ECDLP
\item Problem \ref{prob: seqass}:  EDS Association for non-perfectly periodic sequences
\item Problem \ref{prob: seqres}:  EDS Residue for non-perfectly periodic sequences
\item Problem \ref{prob: edsdlp} ($s=3$):  Width 3 EDS Discrete Log for perfectly periodic sequences
\end{enumerate}
\end{theorem}

Section \ref{sec: hardprobs} relates Problems \ref{prob: edsdlp} and \ref{prob: seqass} to the ECDLP.  Section \ref{sec: seqass} expands on Problem \ref{prob: seqass}.  Sections \ref{sec: seqres1} and \ref{sec: seqres2} discuss Problem \ref{prob: seqres}.  Section \ref{sec: edsdlp} remarks on Problem \ref{prob: edsdlp}. Section \ref{sec: equiv} proves Theorem \ref{mainthm}.  The relation with the MOV and Frey-R\"{u}ck attacks is discussed in Section \ref{sec: fqdlp}.

The authors would like to thank the referees for their helpful suggestions.



\section{Background on Elliptic Nets}
\label{sec: ellnets}

In this section we state the background definitions and results on elliptic divisibility sequences and elliptic nets that are needed for the rest of the paper.  For details and examples, see \cite{Sta4}\cite{Sta1}\cite{Sta3}.

\begin{definition}[{Stange \cite[Def. 2.1]{Sta4}\cite[Def. 3.1.1]{Sta1}}]
\label{def: ellnet}
Let $K$ be a field, $n > 0$ and integer.  An \emph{elliptic net} is any map $W: \Z^n \rightarrow K$ such that the following recurrence holds for all $p$, $q$, $r$, $s \in \Z^n$:
\begin{multline}
\label{eqn: ellrec}
\wvec{p + q + s} \wvec{p- q}\wvec{r+ s}\wvec{r} \\
+ \wvec{q+ r+ s}\wvec{ q- r}\wvec{ p+ s}\wvec{ p} \\
+ \wvec{ r + p+ s} \wvec{ r- p} \wvec{ q+ s}\wvec{ q} = 0
\end{multline}
We refer to $n$ as the \emph{rank} of the elliptic net.  An elliptic net of rank one is called an \emph{elliptic divisibility sequence}.
\end{definition}

One always has $W(-\mathbf{v}) = -W(\mathbf{v})$ and $W(\mathbf{0}) = 0$, and a restriction of an elliptic net to a sublattice of $\Z^n$ is again an elliptic net.  The important fact for our purposes is that any elliptic curve $E$ over $K$ and points $P_1, \ldots, P_n \in E(K)$ gives rise to a unique elliptic net $W_{E,P_1,\ldots,P_n} : \Z^n \rightarrow K$.  The principal theorem is as follows.

\begin{theorem}[{Stange \cite[Thm. 6.1]{Sta4}\cite[Thm. 7.1.1]{Sta1}}]
Let $n > 0$ be an integer.  Let $$E: f(x, y) = y^2 + \alpha_1 xy + \alpha_3 y - x^3 - \alpha_2 x^2 - \alpha_4 x - \alpha_6 = 0$$
be an elliptic curve defined over a field $K$.  Let $\e_i$ be the $\it{i^{th}}$ standard basis vector. For all $\mathbf{v} \in \Z^n$, there are functions $\Psi_\mathbf{v}: E^n \rightarrow K$ in the ring
$$  \left. \Z[\alpha_1, \alpha_2, \alpha_3, \alpha_4, \alpha_6][x_i,y_i]_{i=1}^n\left[(x_i - x_j)^{-1}\right]_{\substack{1 \leq i < j \leq n }} \right/ \left< f(x_i, y_i) \right>_{i=1}^n \subset K(E),$$
such that
\begin{enumerate}
\item $W(\mathbf{v}) = \Psi_{\mathbf{v}}$ satisfies the recurrence \eqref{eqn: ellrec}.
\item $\Psi_{\mathbf{v}} = 1$ whenever $\mathbf{v} = \e_i$ for some $1 \leq i \leq n$ or $\mathbf{v} = \e_i + \e_j$ for some $1 \leq i < j \leq n$.
\item $\Psi_\mathbf{v}$ vanishes at $\mathbf{P} = (P_1, \ldots, P_n) \in E^n$ if and only if $\mathbf{v} \cdot \mathbf{P} = \mathcal{O}$ on $E$ (and $\mathbf{v}$ is not one of the vectors specified in 2).
\end{enumerate}
\end{theorem}

In the case of rank $n=1$, the $\Psi_\mathbf{v}$ are the familiar \emph{division polynomials} of an elliptic curve \cite[p. 105]{Sil1}.
Since the $\Psi_\mathbf{v}$ satisfy the elliptic net recurrence \eqref{eqn: ellrec}, we may make the following definition.

\begin{definition}[{Stange \cite[Def. 6.1]{Sta4}\cite[Def. 7.2.1]{Sta1}}]
For any elliptic curve $E$ defined over $K$ and non-zero points $P_1, \ldots, P_n \in E(K)$ such that no two are equal or inverses (or, if $n=1$, $P_1$ is not a $2$- or $3$-torsion point), the map $W_{E,P_1, \ldots, P_n}: \Z^n \rightarrow K$ defined by $$W_{E,P_1,\ldots, P_n}(\mathbf{v}) = \Psi_\mathbf{v}(P_1, \ldots, P_n)$$ is an elliptic net called the \emph{elliptic net associated to $E, P_1, \ldots, P_n$}.
\end{definition}

Nearly all elliptic nets arise in this way (see \cite{Sta4}\cite{Sta1}).  For the remainder of this article, any elliptic net or elliptic divisibility sequence will be assumed to have this form.

Elliptic nets or elliptic divisibility sequences are arrays or sequences of values of $K$.  The zeroes in this array are particularly important.

\begin{definition}
The zeroes of an elliptic divisibility sequence or elliptic net appear as a sublattice of the lattice of indices.  We call this sublattice the \emph{lattice of zero-apparition}.  In the case of a sequence, this sublattice is specified by a single positive integer -- the smallest positive index of a vanishing term -- and this number is called the \emph{rank of zero-apparition}.
\end{definition}

The rank of zero-apparition of an elliptic divisibility sequence associated to a point $P$ will equal the order of the point $P$.  In the case of an array associated to points $P_1, \ldots, P_n$, the zeroes $(v_1, \ldots, v_n)$ correspond to linear combinations $\mathbf{v} \cdot \mathbf{P}$ that vanish.  

Suppose $T: \Z^s \rightarrow \Z^t$ is a $\Z$-linear transformation.  The following theorem relates the elliptic net associated to $\mathbf{P} \in E^s$ to that associated to $T(\mathbf{P}) \in E^t$.

\begin{theorem}[{Stange \cite[Prop. 5.6]{Sta4}\cite[Thm. 6.2.3]{Sta1}}]
\label{thm: equiv}
Let $T$ be any $t \times s$ integral matrix.  Let $\mathbf{P} \in E^s$ and $\mathbf{v} \in \Z^t$.  Then
\begin{multline}
\label{eqn: equiv}
W_{E,\mathbf{P}}(T^{tr}(\mathbf{v})) = W_{E,T(\mathbf{P})}(\mathbf{v}) \\ \times \prod_{i=1}^{t}W_{E,\mathbf{P}}(T^{tr}(\mathbf{e}_i))^{v_i^2 - v_i\left(\sum_{j\neq i}v_j\right)}
\prod_{1 \leq i < j \leq t} W_{E,\mathbf{P}}(T^{tr}(\mathbf{e}_i+\mathbf{e}_j))^{v_iv_j}
\end{multline}
\end{theorem}

This has several useful corollaries.  For proofs see the cited references.

\begin{theorem}[{Ward \cite[Thm. 8.1]{War}, Stange \cite[Thm. 10.2.3]{Sta1}\cite{Sta5}}]
\label{thm: perrank1}
Suppose that $W_{E,P}(m) = 0$. Then for all $l, v \in \Z$, we have
\[
W_{E,P}(lm+v) = W_{E,P}(v) a^{vl}b^{l^2}
\]
where
\[
a = \frac{W_{E,P}(m+2)}{W_{E,P}(m+1)W_{E,P}(2)}, \qquad
b = \frac{W_{E,P}(m+1)^2W_{E,P}(2)}{W_{E,P}(m+2)}.
\]
Furthermore, $a^m = b^2$.  Therefore, there exists an $\alpha \in \bar K$, the algebraic closure of $K$, such that $\alpha^2 = a$ and $\alpha^m = b$, and so
\[
W_{E,P}(lm+v) = W_{E,P}(v) \alpha^{(lm+v)^2-v^2}.
\]
\end{theorem}

\begin{theorem}[{Stange \cite[Thm. 10.2.3]{Sta1}\cite{Sta5}}]
\label{thm: perrankn}
Suppose $\mathbf{r} = (r_1,r_2) \in \Z^2$ is such that \\$W_{E,P,Q}(\mathbf{r})~=~0$.
For $l \in \Z$ and $\mathbf{v} = (v_1,v_2) \in \Z^2$
we have
\[
W_{E,P,Q}(l\mathbf{r}+\mathbf{v}) = W_{E,P,Q}(\mathbf{v}) 
a_\mathbf{r}^{lv_1}b_\mathbf{r}^{lv_2}c_\mathbf{r}^{l^2}
\]
where
\[ a_\mathbf{r} = \frac{W(r_1+2,r_2)}{W(r_1+1,r_2)w(2,0)}, \; b_\mathbf{r} = \frac{W(r_1,r_2+2)}{W(r_1,r_2+1)W(0,2)}, \; c_\mathbf{r} = \frac{W(r_1+1,r_2+1)}{a_\mathbf{r}b_\mathbf{r}W(1,1)}.
\]
\end{theorem}

\section{Perfectly Periodic Sequences and Nets}
\label{sec: perfper}
\begin{definition}  An elliptic divisibility sequence is called \emph{perfectly periodic} if it is periodic with respect to its rank of zero-apparition.  An elliptic net is called \emph{perfectly periodic} if it is periodic with respect to its lattice of zero-apparition.
\end{definition}


\begin{definition}  Let $f: \Z^n \rightarrow K^*$ be a quadratic function, and $k \in K^*$ a constant.  Two elliptic nets $W$ and $W'$ are called \emph{equivalent} if $W'(\mathbf{v}) = kf(\mathbf{v})W(\mathbf{v})$.
\end{definition}

As an example, let $W$ be an elliptic divisibility sequence with rank of zero-apparition $m$.  In one variable ($n=1$), quadratic functions to $K^*$ have the form $f(n) = \alpha^{n^2}$ for some $\alpha \in K^*$.  Suppose we use $\alpha$ as defined by Theorem \ref{thm: perrank1}, i.e. $\alpha^2 = a, \alpha^m = b$, and let take $k = \alpha^{-1}$.  Then $W'(n) = \alpha^{n^2-1}W(n)$, and this sequence is perfectly periodic.  Suppose that $K=\F_q$ and $\gcd(q-1,m)=1$.  In this case the conditions of Theorem \ref{thm: perrank1} determine such an $\alpha$ uniquely, and it lies in $K$.  Otherwise (if $\gcd(q-1,m) \neq 1$), two such $\alpha$'s will exist, equal up to sign.  The two resulting perfectly periodic sequences will be equal at even-indexed locations and equal up to sign at odd-indexed locations.  

The moral of the last paragraph is that any elliptic divisibility sequence is equivalent to a perfectly periodic one.  We can give an explicit expression for such a perfectly periodic sequence.

\begin{theorem}
\label{thm: perfper}
Let $K$ be a finite field of $q$ elements, and $E$ an elliptic curve defined over $K$.  For all points $P \in E$ of order relatively prime to $q-1$ and greater than $3$, define
\begin{equation}
\label{eqn: phip}
\phi(P) = \left( \frac{W_{E,P}(q-1)}{W_{E,P}(q-1+\operatorname{ord}(P))} \right)^{\frac{1}{\operatorname{ord}(P)^2}}.
\end{equation}
For a point $P$ of order relatively prime to $q-1$ and greater than $3$, the sequence $\phi([n]P)$ is a perfectly periodic elliptic divisibility sequence equivalent to $W_{E,P}(n)$.  Specifically,
\begin{equation}
\label{eqn: perfper}
\phi([n]P) = \phi(P)^{n^2-1}W_{E,P}(n).
\end{equation}
More generally, let $\mathbf{P} \in E(K)^n$ be a collection of nonzero points, no two equal or inverses, and all elements of a single cyclic group and having order greater than $3$.  The $n$-array $\phi(\mathbf{v} \cdot \mathbf{P})$ (as $\mathbf{v}$ ranges over $\mathbb{Z}^n$) forms a perfectly periodic elliptic net equivalent to $W_{E,\mathbf{P}}(\mathbf{v})$.  Specifically,
\[
\phi(\mathbf{v} \cdot \mathbf{P}) = W_{E,\mathbf{P}}(\mathbf{v}) \prod_{i=1}^n \phi(P_i)^{v_i^2 - v_i\left(\sum_{j\neq i}v_j\right)}\prod_{1 \leq i < j \leq n} \phi(P_i + P_j)^{v_iv_j}.
\]
\end{theorem}

\begin{proof}
The proof uses Theorem \ref{thm: equiv}.  We will demonstrate the method of proof in the rank one case before proceeding to the general case.  Take $T = (l)$, so
$$W_{E,[l]P}(n)W_{E,P}(l)^{n^2} = W_{E,P}(nl).$$
By symmetry,
$$W_{E,[n]P}(l)W_{E,P}(n)^{l^2} = W_{E,P}(nl).$$
Let $m = \operatorname{ord}(P)$.  Thus, combining the above and using $l = q-1$ and $q-1+m$ in turn,
\begin{eqnarray*}
\frac{W_{E,[n]P}(q-1)W_{E,P}(n)^{(q-1)^2}}
{W_{E,P}(q-1)^{n^2}}&=& W_{E,[q-1]P}(n) = W_{E,[q-1+m]P}(n) \\
&=& \frac{W_{E,[n]P}(q-1+m)W_{E,P}(n)^{(q-1+m)^2}}
{W_{E,P}(q-1+m)^{n^2}}
\end{eqnarray*}
Rearranging, 
\[
\phi([n]P) = \phi(P)^{n^2-1}W_{E,P}(n).
\]
Therefore, $\phi([n]P)$ is an elliptic divisibility sequence.  By definition, $\phi([n]P)$ has period $\operatorname{ord}(P)$ which is equal to the rank of apparition of $W_{E,P}$ and $\phi([n]P)$.  So $\phi([n]P)$ is perfectly periodic.

For the rank $n$ case, let $m$ be the order of the cyclic group containing all the points under consideration.  In Theorem \ref{thm: equiv}, let $t=1$ and $s=n$ and
take $T = ( v_1 \quad v_2\quad  v_3 \quad \cdots \quad v_n )$ to obtain
\[
W_{E,\mathbf{P}}(l\mathbf{v}) = W_{E,\mathbf{v} \cdot \mathbf{P}}(l) W_{E,\mathbf{P}}(\mathbf{v})^{l^2}.
\]
Now take $t=s=n$ in Theorem \ref{thm: equiv} , and $T = l\rm{Id}_n$ to obtain
\[
W_{E,\mathbf{P}}(l\mathbf{v}) = W_{E,l\mathbf{P}}(\mathbf{v})  \prod_{i=1}^n W_{E,\mathbf{P}}(le_i)^{v_i^2 - v_i(\sum_{j\neq i}v_j)}\prod_{1 \leq i < j \leq n} W_{E,\mathbf{P}}(le_i + le_j)^{v_iv_j}. 
\]
Note that
\[
W_{E,\mathbf{P}}(le_i) = W_{E,P_i}(l), \qquad W_{E,\mathbf{P}}(le_i + le_j) = W_{E,P_i+P_j}(l).
\]
Combining the above, we have
\[
W_{E,l\mathbf{P}}(\mathbf{v})  = \frac{
 W_{E,\mathbf{v} \cdot \mathbf{P}}(l) W_{E,\mathbf{P}}(\mathbf{v})^{l^2} 
 }{
   \prod_{i=1}^n W_{E,P_i}(l)^{v_i^2 - v_i(\sum_{j\neq i}v_j)}\prod_{1 \leq i < j \leq n} W_{E,P_i+P_j}(l)^{v_iv_j}
   }. 
\]
Comparing this in the case of $l = q-1$ and $l=q-1+m$ gives the required result, as before. \end{proof}

In light of this theorem we will use the convenient notation
\[
\widetilde{W}_{E,P}(n) = \phi([n]P).
\]
and call this the \emph{perfectly periodic elliptic divisibility sequence associated to $E$ and $P$}.  The attractive property of a perfectly periodic sequence is formula \eqref{eqn: phip}:  $\widetilde{W}_{E,P}(n)$ can be calculated \emph{as a function of the point $[n]P$ on the curve} without knowledge of $n$.

\begin{corollary}
\label{cor: perfperorder}
Suppose that $E$ is an elliptic curve over a field $K = \Fq$ and $P \in E(K)$ is of order $m \geq 4$.  The period of the sequence $W_{E,P}$ is $m\operatorname{ord}_{K^*}(\phi(P))$.
\end{corollary}

\begin{proof}
First, $\phi([n]P)$ has period exactly $m$.  Since, if the period were $m'<m$, then $W_{E,P}(m') = 0$, a contradiction.  The result now follows from equation \eqref{eqn: perfper}.
\end{proof}

The ratio between the period and the rank of zero-apparition, which we've demonstrated to be $\operatorname{ord}_{K^*}(\phi(P))$, is called $\tau$ by Morgan Ward \cite[Thm. 11.1]{War}.

\section{The Hard Problems}
\label{sec: hardprobs}

As we have seen, elliptic nets are closely related to the points on an elliptic curve.  In this section, we will see specifically how to compute them, and how they relate, algorithmically, to the points.

The choice of segment $0~<~k~<~\operatorname{ord}(P)$ is not crucial in Problem \ref{prob: seqass} (EDS Association):  it could be restated for any segment $ i~\operatorname{ord}(P)~<~k~<~(i~+~1)~\operatorname{ord}(P)$.  This problem is trivial for a perfectly periodic sequence or net (since $\widetilde W(k) = \phi(Q)$ is computable in log $q$ time).  For the non-perfectly periodic case, the problem appears to be much harder.  As for Problem \ref{prob: edsdlp} (EDS Discrete Log), on the other hand, for non-perfectly periodic elliptic divisibility sequences, it can be solved by computing an $\F_q^*$ discrete log.  For this problem, it is the case of perfect periodicity that seems very difficult.

We will see that these hard problems are related according to the following diagram.

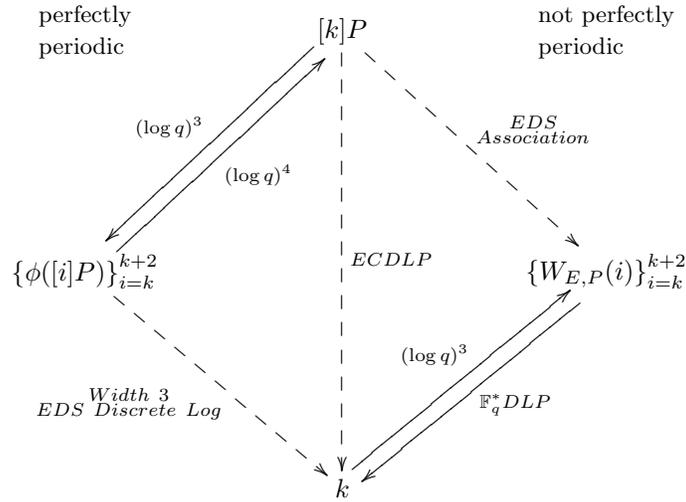
\[
\xymatrix@R=5.5pc @C=4.5pc{
 {\begin{array}{l} \mbox{\small perfectly}\\ \mbox{\small periodic} \end{array}} & [k]P \ar@<-0.7ex>[dl]_{(\log q)^3} \ar@{-->}[dr]^{\substack{EDS\\ Association}} \ar@{-->}[dd]^{ECDLP}&  {\begin{array}{l} \mbox{\small not perfectly}\\ \mbox{\small periodic} \end{array}} \\
{ \left\{ \phi([i]P) \right\}_{i=k}^{k+2} } \ar@{-->}[dr]_{\substack{Width\mbox{ }3 \\ EDS\mbox{ }Discrete\mbox{ }Log}} \ar@<-0.7ex>[ur]_{(\log q)^4} & & \left\{ W_{E,P}(i) \right\}_{i=k}^{k+2} \ar@<0.7ex>[dl]^{\F_q^* DLP}\\
& k \ar@<0.7ex>[ur]^{(\log q)^3}  & 
}
\]

We demonstrate the complexity of solving the problems associated to the solid lines in the following series of theorems.  The solid line labelled $\Fq^*$DLP has the complexity of a discrete logarithm problem in $\Fq^*$ (this is sub-exponential by index calculus).  No sub-exponential algorithms are known for the dotted lines.

Since our concern is polynomial time vs. non-polynomial time, in the following we assume naive arithmetic in $\Fq$, i.e. we bound the time to do basic $\Fq$ operations by $O((\log q)^2)$ for simplicity.

\begin{lemma}
\label{lemma: qxcoord}
Let $E$ be an elliptic curve defined over $K$, and $P \in E(K)$ be a point of order not less than $4$.  The $x$-coordinate of $[n]P$, $x([n]P)$, can be calculated in $O((\log q)^2)$ time from the three terms $W_{E,P}(n-1)$, $W_{E,P}(n)$, and $W_{E,P}(n+1)$ or from the three terms $\widetilde{W}_{E,P}(n-1)$, $\widetilde{W}_{E,P}(n)$, and $\widetilde{W}_{E,P}(n+1)$.
\end{lemma}
\begin{proof}
See \cite[Lemma 6.2.2]{Sta1} for the following identity:
\begin{equation}
\label{eqn: qxcoord}
\frac{ W_{E,P}(n-1)W_{E,P}(n+1) }{ W_{E,P}(n)^2 } = x(P) - x([n]P).
\end{equation}
The left-hand side of \eqref{eqn: qxcoord} is invariant under equivalence, and so the same calculation applies if we put tilde's on the $W$'s.
\end{proof}

\begin{theorem}[{Shipsey \cite[Thm 3.4.1]{Shi}}]
\label{thm: shipsey}
Let $E$ be an elliptic curve over $K$, and $P \in E(K)$ a point of order not less than $4$.  Given a value $t$, the term $W_{E,P}(t)$ in the elliptic divisibility sequence associated to $E, P$ can be calculated in $O((\log t) (\log q)^2)$ time.
\end{theorem}

\begin{proof}
For completeness, we give a simplified version of Shipsey's algorithm here.  Following Shipsey, denote by $\left< W_{E,P}(k) \right>$ the segment or \emph{block centred at $k$} of eight terms $W_{E,P}(k-3)$, $W_{E,P}(k-2)$, $\ldots$, $W_{E,P}(k+3)$, $W_{E,P}(k+4)$ of the sequence.  The block centred at $t$ can be calculated from the block centred at $1$ via a double-and-add algorithm based on an addition chain for $t$.  The calculation of the new block from the previous depends on two instances of the recurrence (one such calculation for each term of the new block):
\begin{align}
W(2i-1,0) &= W(i+1,0)W(i-1,0)^3- W(i-2,0)W(i,0)^3 \enspace , \notag \\
W(2i,0)&= \left(W(i,0)W(i+2,0)W(i-1,0)^2 \right. \notag \\ & \hspace{4em}\left. - W(i,0)W(i-2,0)W(i+1,0)^2\right)/W(2,0) \enspace . \notag
\end{align}
To begin we must calculate the block centred at 1.  Recalling that $W(0) = 0$, $W(1)=1$ and $W(-n) = -W(n)$, we must calculate $W(i)$ for $i=2,3,4$.  Precise formulae in terms of the coordinates of $P$ and the Weierstrass coefficients for $E$ can be found in \cite[p. 105]{Sil1} or for long Weierstrass equations in \cite[p. 80]{FreLan}.  This algorithm takes $O(\log t)$ steps, each of which involves a fixed number of $\Fq^*$ multiplications and additions, which take $O((\log q)^2)$ time at worst.
\end{proof}

\begin{theorem}
\label{thm: calcphi}
Let $E$ be an elliptic curve over $\Fq$, and $P \in E(\Fq)$ a point of order relatively prime to $q-1$ and greater than $3$.  Given a point $Q = [k]P$, the term $\phi(Q) = \widetilde W_{E,P}(k)$ can be calculated in $O((\log q)^3)$ time without requiring knowledge of $k$.
\end{theorem}

\begin{proof}
We use equation \eqref{eqn: phip}.  Using Theorem \ref{thm: shipsey} to calculate the ratio of terms inside the parentheses takes $\log (q-1+\operatorname{ord}(Q))+\log(q-1)$ steps.  Since $\operatorname{ord}(Q)$ is on the order of $q$, this is $O((\log q)^3)$ time at worst.  The other necessary operation in \eqref{eqn: phip} is to find the inverse of $\operatorname{ord}(Q)^2$ modulo $q-1$, and to raise to that exponent.  Both these are also $O(\log q)$ operations.
\end{proof}

\begin{theorem}
\label{thm: calcfromphi}
Let $E$ be an elliptic curve over $\Fq$, and $P \in E(\Fq)$ a point of order relatively prime to $q-1$ and greater than $3$.  Given the $\widetilde W_{E,P}(k)$, $\widetilde W_{E,P}(k+1)$ and $\widetilde W_{E,P}(k+2)$, the point $Q = [k]P$ can be calculated in probabilistic $O((\log q)^4)$ time without requiring knowledge of $k$.
\end{theorem}

\begin{proof}
Calculate $x([k+1]P)$ by Lemma \ref{lemma: qxcoord}.  We can calculate the corresponding possible values for $y$ in probabilistic time $O((\log q)^4)$ \cite[\S 7.1-2]{BacSha}.  To determine which of the two points with this $x$-coordinate is actually $[k+1]P$, first take one of the two candidate points, and proceed on the assumption that it is $[k+1]P$.  Using the addition formula for elliptic curves, calculate $x([k+1]P + P) = x([k+2]P)$.  Compare this with \eqref{eqn: qxcoord} to determine $\widetilde W(k+3)$.  Also determine $\widetilde W(k+4)$ in this manner.  Then, if the terms $\widetilde W(k), \ldots, \widetilde W(k+4)$ satisfy the recurrence instance
\[
\widetilde W(k+4)\widetilde W(k) =  \widetilde W(k+1)\widetilde W(k+3)\widetilde W(2)^2 - \widetilde W(3)\widetilde W(1)\widetilde W(k+2)^2,
\]
our assumption about the point we chose is correct.  If this recurrence does not hold, then the point we chose was incorrect, and the other one is the point $[k+1]P$ we seek.  For, it is impossible that both points cause the above equation to be satisfied:  any sequence of four consecutive terms in an elliptic divisibility sequence determines the entire sequence uniquely.  Finally, knowing $[k+1]P$, we can calculate $Q = [k]P = [k+1]P - P$.
\end{proof}

The following theorem is implicit in the work of Shipsey; see Section~\ref{subsec: shipsey} for an explanation. 
\begin{theorem}
\label{thm: redtodlp}
Suppose $P$ has order relatively prime to $q-1$ and greater than $3$, and $\phi(P)$ is a primitive root in $\Fq^*$.  Given $W_{E,P}(k), W_{E,P}(k+1), W_{E,P}(k+2)$, where it can be assumed that $0 < k < ord(P)$, calculating $k$ can be reduced to a single discrete logarithm in $\F_q^*$ in probabilistic $O((\log q)^4)$ time.
\end{theorem}


\begin{proof}
We can deduce the $x$-coordinate of the point $Q = [k]P$ by Lemma \ref{lemma: qxcoord}.  Compute the two corresponding $y$-coordinates, which takes probabilistic time $O((\log q)^4)$ \cite[\S 7.1-2]{BacSha}. Choosing one of the two possible $y$-coordinates, we have either $Q = [k]P$ or $Q= [-k]P$.  To determine which is correct, use the trick of the proof of Theorem \ref{thm: calcfromphi}.  Suppose it is the former; then, from Theorem \ref{thm: perfper}, we have
\begin{equation}
\label{eqn: 2perfper}
\frac{\phi([k+1]P)}{\phi([k]P)} = \phi(P)^{2k+1}\frac{W_{E,P}(k+1)}{W_{E,P}(k)}.
\end{equation}
So $k$ satisfies an equation of the form $A = B^{2k+1}$ where $A$ and $B$ are known, and $B$ has order $q-1$ by assumption.  Therefore, we are reduced to solving a discrete logarithm of the form $A = B^x$ for $0 \leq x < q-1$, with the understanding that $k$ will be one of $(x-1)/2$ or $(x+q-1)/2$.  (In fact, if $q-1 < m$, there may be at most two other possible values of $k$ to check:  the above values shifted by $q-1$.) 
\end{proof}

\begin{remark}
\label{remark: mq-1}
Let $m = \operatorname{ord}(P)$.  Suppose that $\gcd(m, q-1) = 1$.  As an integer $k$ ranges over representatives of a single coset in $\Z / m\Z$, it ranges over all possible cosets of $\Z / (q-1)\Z$.  Therefore, we cannot expect to find the set of $k$ such that $Q = [k]P$ (i.e. a coset in $\Z / m\Z$) by solving an equation of the form $A = B^k$ in $\F_q^*$ (i.e. solving modulo $q-1$).  One solution to this problem is to attempt to solve for an \emph{integer} $k$ (instead of a coset) -- say, for example, the smallest non-negative $k$ with $Q = [k]P$.  This is in essence what the preceeding theorem does.  With this in mind, we set some terminology.
\end{remark}

\begin{definition}
Let $Q$ be a multiple of $P$ on an elliptic curve $E$.  The \emph{minimal multiplier} of $Q$ with respect to $P$ is the smallest non-negative value of $k$ such that $Q = [k]P$.
\end{definition}

Note that the minimal multiplier satisfies $0 \leq k < \operatorname{ord}(P)$.

\section{$\F_q^*$ Discrete Logarithm, The Tate Pairing and MOV/Frey-R\"{u}ck Attack}
\label{sec: fqdlp}

Theorem \ref{thm: redtodlp} uses terms of the elliptic divisibility sequence to give a discrete logarithm problem in $\F_q^*$.  We demonstrate some variations on this theme, and relate these types of equations to the Tate pairing, and to an ECDLP attack given by Shipsey \cite{Shi}.

\subsection{An $\F_q^*$ DLP equation of the form $A = B^k$ from periodicity properties}
\label{subsec: eqper}

The $\F_q^*$ DLP equations we consider are consequences of Theorem \ref{thm: equiv}, but many can be conveniently understood in terms of its corollary Theorem \ref{thm: perrankn}.  The following example involves the terms $W_{E,P}(k)$ and $W_{E,P}(k+1)$, and requires knowledge of $Q = [k]P$.  The following  diagram is suggestive for the discussion.

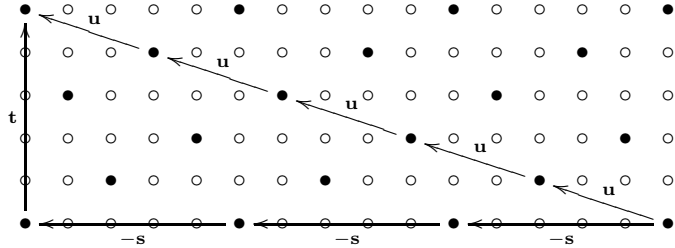
\[
\SelectTips{xy}{705pt}
\xymatrix@!=0.001pc@R=1.35pc @C=1.35pc{
{\bullet} & {\circ} & {\circ} & {\circ} & {\circ} & {\bullet} & {\circ} & {\circ} &
{\circ} & {\circ} & {\bullet} & {\circ} & {\circ} & {\circ} & {\circ} & {\bullet} \\
{\circ} & {\circ} & {\circ} & {\bullet} \ar[lllu]_{\mathbf{u}} & {\circ} & {\circ} & {\circ} & {\circ} &
{\bullet} & {\circ} & {\circ} & {\circ} & {\circ} & {\bullet} & {\circ} & {\circ} \\
{\circ} & {\bullet} & {\circ} & {\circ} & {\circ} & {\circ} & {\bullet} \ar[lllu]_{\mathbf{u}} & {\circ} &
{\circ} & {\circ} & {\circ} & {\bullet} & {\circ} & {\circ} & {\circ} & {\circ} \\
{\circ} & {\circ} & {\circ} & {\circ} & {\bullet} & {\circ} & {\circ} & {\circ} &
{\circ} & {\bullet} \ar[lllu]_{\mathbf{u}} & {\circ} & {\circ} & {\circ} & {\circ} & {\bullet} & {\circ} \\
{\circ} & {\circ} & {\bullet} & {\circ} & {\circ} & {\circ} & {\circ} & {\bullet} &
{\circ} & {\circ} & {\circ} & {\circ} & {\bullet} \ar[lllu]_{\mathbf{u}} & {\circ} & {\circ} & {\circ} \\
{\bullet} \ar[uuuuu]^{\mathbf{t}} & {\circ} & {\circ} & {\circ} & {\circ} & {\bullet}\ar[lllll]^{-\mathbf{s}} & {\circ} & {\circ} &
{\circ} & {\circ} & {\bullet} \ar[lllll]^{-\mathbf{s}}& {\circ} & {\circ} & {\circ} & {\circ} & {\bullet} \ar[lllll]^{-\mathbf{s}} \ar[lllu]_{\mathbf{u}}
}
\]
In this picture of $\Z^2$, $\mathbf{u} = (-3,1)$, $\mathbf{s}=(5,0)$ and $\mathbf{t}=(0,5)$.  Vectors $\mathbf{u}$ and $\mathbf{s}$ generate the lattice of zero-apparition $\Lambda$ for some elliptic net $W$ associated to points $P$ and $Q= [3]P$ of order $5$.  The vector $\mathbf{t}$ is also in $\Lambda$.  One coset of $\Z^2$ modulo $\Lambda$ is shown as the solid discs.

Theorem \ref{thm: perrankn} shows the transformation relative to translation by a vector $\r \in \Lambda$:  it relates $W(\mathbf{v} + \mathbf{r})$ to $W(\mathbf{v})$ for each $\mathbf{v}$.  This Lemma can be applied repeatedly, and different `paths' from one point to another must agree.  In the picture above, the translation property which relates $W(\mathbf{v} + (-15,5))$ to $W(\mathbf{v})$ can be calculated by applying the transformation associated to $\mathbf{u}$ five times (the diagonal path) or by applying the transformation associated to $-\mathbf{s}$ three times followed by that associated to $\mathbf{t}$ once (the sides of the triangle).

In the general case, we have $Q = [k]P$.  Then the lattice of zero-apparition $\Lambda$ for $W = W_{E,P,Q}$ includes vectors $\mathbf{u} = (-k,1)$, $\mathbf{s} = (m, 0)$ and $\mathbf{t} = (0,m)$.  Suppose $\mathbf{r} = (r_1, r_2)$ is an element of $\Lambda$ for $W = W_{E,P,Q}$.  By Theorem \ref{thm: perrankn}, we have for all $l \in \Z$ and $\mathbf{v} \in \Z^2$,
\begin{equation}
\label{eqn: sym2}
W(l\mathbf{r}+\mathbf{v}) = W(\mathbf{v}) a_\mathbf{r}^{lv_1} b_\mathbf{r}^{lv_2} c_\mathbf{r}^{l^2}
\end{equation}
where
\[
a_\mathbf{r} = \frac{W(r_1+2,r_2)}{W(r_1+1,r_2)W(2,0)}, \; b_\mathbf{r} = \frac{W(r_1,r_2+2)}{W(r_1,r_2+1)W(0,2)}, \; c_\mathbf{r} = \frac{W(r_1+1,r_2+1)}{a_\mathbf{r}b_\mathbf{r} W(1,1)}.
\]

We expect appropriate relationships between $a_\mathbf{u}$, $b_\mathbf{u}$, $c_\mathbf{u}$, $a_\mathbf{s}$, $b_\mathbf{s}$, etc.  The $\Fq^*$ DLP equation we seek is one such relationship.  We have
\[
a_\mathbf{s} = \frac{W(m+2,0)}{W(m+1,0)W(2,0)}, \; a_\mathbf{t} = \frac{W(2,m)}{W(1,m)W(2,0)}, \; a_\mathbf{u} = \frac{W(2-k,1)}{W(1-k,1)W(2,0)}.
\]
For each $i \in \Z$, we apply \eqref{eqn: sym2} to obtain
\begin{equation}
\label{eqn: aui}
\frac{W(-ik+1,i-1)W(0,-1)}{W(1,-1)W(-ik,i-1)} = a_\mathbf{u}^i
\end{equation}
Set $i=m$ in \eqref{eqn: aui}, and apply \eqref{eqn: sym2} four times:
\begin{eqnarray*}
a_\mathbf{u}^m &=& \frac{W(-mk+1,m-1)W(0,-1)}{W(1,-1)W(-mk,m-1)} \\
&=& \textstyle{ \left( \frac{W(-mk+1,m-1)}{W(-mk+1,-1)} \right) \left( \frac{W(-mk+1,-1)}{W(1,-1)} \right) \left( \frac{W(0,-1)}{W(-mk,-1)} \right) \left( \frac{W(-mk,-1)}{W(-mk,m-1)} \right) }\\
&=& \frac{
a_\mathbf{t}^{-mk+1} b_\mathbf{t}^{-1} c_\mathbf{t}^1 a_\mathbf{s}^{-k} b_\mathbf{s}^{k} c_\mathbf{s}^{k^2}
}{
a_\mathbf{t}^{-mk} b_\mathbf{t}^{-1} c_\mathbf{t}^1 a_\mathbf{s}^0 b_\mathbf{s}^{k} c_\mathbf{s}^{k^2}
 } = a_\mathbf{t} a_\mathbf{s}^{-k}
\end{eqnarray*}
Setting $i=1$ in \eqref{eqn: aui}, we obtain an expression
\[
a_\mathbf{u} = \frac{W(-k+1,0)W(0,-1)}{W(1,-1)W(-k,0)} = - \frac{W_{E,P}(k-1)}{W_{E,P}(k) W(1,-1)}
\]
which, when substituted into the last calculation, yields
\begin{equation}
\label{eqn: tridlp}
\left( \frac{ W(m+1,0)W(2,0) }{ W(m+2,0) } \right)^k
= 
\left( \frac{ W_{E,P}(k-1) }{ W_{E,P}(k) } \right)^m \left( - \frac{  W(1,m) W(2,0) }{W(2,m)W(1,-1)^{m}} \right).
\end{equation}

\subsection{An $\F_q^*$ DLP equation from Shipsey's Thesis}
\label{subsec: shipsey}

The possibility of such an equation was observed by Rachel Shipsey in her thesis \cite[(6.3)]{Shi}.  She uses one-dimensional periodicity properties to derive the following equation:
\begin{equation}
\label{eqn: shipseyrank1}
\frac{ W_{E,P}((m+1)(k+1)) W_{E,P}(k) } {W_{E,P}((m+1)k)W_{E,P}(k+1)}
= W_{E,P}(m+1)^{2k+1}
\end{equation}
Shipsey then argues that without knowledge of $k$ the left hand side can be calculated up to a factor of $$\left( \frac{ W_{E,P}(k) }{W_{E,P}(k-1)} \right)^{m(m+2)}.$$  This is very much of the same spirit as equation \eqref{eqn: tridlp}, and in fact, Theorem \ref{thm: equiv} can be used to rewrite \eqref{eqn: shipseyrank1} in this form:
\begin{equation}
\label{eqn: shipsey}
\frac{W_{E,P,Q}(m+1,m+1)}{W_{E,P,Q}(0,m+1)} \left( \frac{W_{E,P}(k+1)}{W_{E,P}(k)} \right)^{m(m+2)} = W_{E,P}(m+1)^{2k+1}.
\end{equation}
By Lemma \ref{lemma: qxcoord}, knowledge of $Q, W_{E,P}(k), W_{E,P}(k-1)$ determines $W_{E,P}(k+1)$, and so this is very much equivalent to Shipsey's analysis.  Note that the unknown terms in \eqref{eqn: shipsey} are raised to the exponent $m+2$.  At first blush, this may appear to lead to an ECDLP attack for $q-1=m+2$ (where the unknown terms will disappear).  However, this is not allowed by Remark \ref{remark: mq-1}.  In fact, it turns out that if $q-1=m+2$, then $W_{E,P}(m+1) = 1$ (this eventually follows from Theorem \ref{thm: equiv} also).

\subsection{$\F_q^*$ DLP equations and the Tate pairing}
\label{subsec: tate}

Choose $m \in \Z^+$.  Let $E$ be an elliptic curve defined over a finite field $K$ containing the $m$-th roots of unity.  Suppose $P \in E(K)[m]$ and $Q \in E(K)/mE(K)$.  Since $P$ is an $m$-torsion point, $m(P) - m(\mathcal{O})$ is a principal divisor, say $\operatorname{div}(f_P)$.  Choose another divisor $D_Q$ defined over $K$ such that $D_Q \sim (Q) - (\mathcal{O})$ and with support disjoint from $\operatorname{div}(f_P)$.  Then, we may define the Tate pairing
\begin{equation*}
\tau_m : E(K)[m] \times E(K) / mE(K) \rightarrow K^* / (K^*)^m
\end{equation*}
and Weil pairing
\begin{equation*}
e_m : E(K)[m] \times E(K)[m] \rightarrow \mathbb{\mu}_m
\end{equation*}
by
\begin{equation*}
\tau_m(P,Q) = f_P(D_Q), \qquad e_m(P,Q) = f_P(D_Q) f_Q(D_P)^{-1}.
\end{equation*}
Both are non-degenerate bilinear pairings, while the Weil pairing is alternating.  For details, see \cite{DuqFre2}\cite{Gal}.

The Tate pairing and Weil pairing are used in the MOV \cite{MenOkaVan} and Frey-R\"{u}ck \cite{FreRuc} attacks on the ECDLP.  These use the Weil and Tate pairings, respectively, to translate an instance of the ECDLP into an $\Fq^*$ DLP equation, where index calculus methods may be used.  The basic idea, illustrated here for the Tate pairing, is that $Q = [k]P$ implies $\tau_m(Q,S) = \tau_m(P,S)^k$ by bilinearity.  If $S$ can be chosen so that $\tau_m(P,S)$ is non-trivial, and if the Tate pairing takes values in a manageably small finite field, then index calculus methods can be used to determine $k$.  In particular, this attack applies for curves $E$ over $\Fq$ where $m = q-1$.

In \eqref{eqn: shipsey} and \eqref{eqn: tridlp}, all the terms may be calculated from knowledge of $m$, $P$ and $Q$ except for $W_{E,P}(k)$ and $W_{E,P}(k-1)$.  However, notice that these unknown terms are raised to the power $m$.  Therefore, in the case that $m=q-1$, no extra information is needed and the ECDLP is reduced to an $\Fq^*$ DLP; this works in exactly the cases that the MOV or Frey-R\"{u}ck attack applies.

These sorts of `alternate versions' of the MOV/Frey-R\"{u}ck attack do have a relation to the Tate pairing.

\begin{theorem}[{Stange \cite[Thm. 17.2.1]{Sta1}\cite[Thm. 6]{Sta3}}]
\label{thm: tate}
Let $E$ be an elliptic curve, $m \geq 4$, and $P \in E[m]$.  Let $Q, S \in E$ be such that $S \not\in \left\{ \mathcal{O}, Q \right\}$.  Let $W$ be an elliptic net of rank $n$, associated to points $\mathbf{T} \in E(K)^n$.   Let $\mathbf{s}, \mathbf{p}, \mathbf{q} \in Z^n$ be such that
$$ P = \mathbf{p} \cdot \mathbf{T}, \qquad Q = \mathbf{q} \cdot \mathbf{T}, \qquad S = \mathbf{s} \cdot \mathbf{T}. $$
Let $\tau_m: E[m] \times E / mE \rightarrow K^* / (K^*)^m$ be the Tate pairing.  Then
\begin{equation*}
\tau_m(P,Q) = \frac{
W(m\mathbf{p} + \mathbf{q} + \mathbf{s}) W(\mathbf{s})
}{
W(m\mathbf{p} + \mathbf{s}) W(\mathbf{q} + \mathbf{s})
}.
\end{equation*}
\end{theorem}

Now equations \eqref{eqn: tridlp} and \eqref{eqn: shipsey} can be re-written as statements in terms of the Tate pairing.  

{\bf Equation \eqref{eqn: tridlp}:}  Use Theorem \ref{thm: tate} with $\mathbf{p} = (1,0), \mathbf{q} = (-1,0), \mathbf{s} = (2,0)$ for the left-hand side and $\mathbf{p} = (0,1), \mathbf{q} = (-1,0), \mathbf{s} = (2,0)$ for the right.  This rewrites \eqref{eqn: tridlp} as
\begin{equation*}
\tau_m(P,-P)^k = \tau_m(Q,-P).
\end{equation*}

{\bf Equation \eqref{eqn: shipsey}:}  This is somewhat more complicated.  From Theorem \ref{thm: perrank1} with $m=q-1$ and Theorem \ref{thm: tate} with various parameters,
\[
W_{E,P}(m+1)^2 \tau_m(P,P)^{-2} = \left( \frac{W_{E,P}(m+1)^2W_{E,P}(2)}{W_{E,P}(m+2)} \right)^2 = b^2 = a^m = 1,
\]
\begin{align*}
\tau_m(P,Q) &= \frac{ W_{E,P,Q}(m+1,1)W_{E,P,Q}(1,0) }{ W_{E,P,Q}(m+1,0) W_{E,P,Q}(1,1) }, \\ \tau_m(Q,P) &= \frac{ W_{E,P,Q}(1,m+1)W_{E,P,Q}(0,1) }{ W_{E,P,Q}(0,m+1) W_{E,P,Q}(1,1) },
\end{align*}
\begin{equation*}
1 = \tau_m(P,\mathcal{O}) = \tau_m(P,[m]Q) =  \frac{ W_{E,P,Q}(m+1,m+1)W_{E,P,Q}(1,1) }{ W_{E,P,Q}(m+1,1) W_{E,P,Q}(1,m+1) }.
\end{equation*}
All of which, taken together, rewrites \eqref{eqn: shipsey} as
\[
\tau_m(P,Q)\tau_m(Q,P) = \tau_m(P,P)^{2k}.
\] 

Equation \eqref{eqn: perfper} (with $n=k$) does not, however, lend itself to this sort of re-writing in terms of pairings in the case $m=q-1$, as the very definition of $\phi(P)$ requires the assumption that $\gcd(m, q-1) = 1$.  

\section{ECDLP through EDS Association}
\label{sec: seqass}

The previous sections have demonstrated that there are a variety of ways to translate an ECDLP into an $\Fq^*$ DLP.  The $\Fq^*$ DLP equation is in terms of elements of the sequence $W_{E,P}$.  For example in \eqref{eqn: tridlp}, the elements are $W_{E,P}(k)$ and $W_{E,P}(k-1)$.  The problem of finding these terms (with knowledge of $Q = [k]P$ but not $k$) is \emph{EDS Association}.  In this example, however, it is only their quotient that is needed.  Depending on the form of the $\Fq^*$ DLP equation, different information (certain terms or ratios of terms) suffices.  We formalise the most general statement of this in the following theorem.

\begin{proposition}
\label{prop: prod1}
Fix an elliptic curve $E$ defined over $\Fq$, and $P \in E(\Fq)$ of order greater than three and relatively prime to $q-1$.  Suppose $\phi(P)$ has order $q-1$ in $\Fq^*$.  With knowledge of any product
\begin{equation}
\label{eqn: prod1}
\prod_{i=1}^N W_{E,P}(p_i(k))^{e_i}, 
\end{equation}
where the $e_i \in \Z$, and $p_i(x) \in \Z[x]$ of degree at most $D$, and
$t(x) = \sum_{i=1}^N e_i p_i(x)^2$
is a non-constant polynomial of degree at most $2$ in $\Z[x]$, the value of $k$ can be determined in subexponential time in $q$, with constants depending on $D$ and $N$.
\end{proposition}

\begin{proof}
Combine appropriate instances of equation \eqref{eqn: perfper} of Theorem \ref{thm: perfper} in such a way that $t(k)$ satisfies an equation in $\F_q^*$ of the form $A = B^{t(k)}$.  That is, combine one instance for each $n=p_i(k)$ with multiplicities given by the respective $e_i$, and obtain an equation of the form
\[
\frac{\displaystyle{\prod_{i=1}^N \widetilde W_{E,P}(p_i(k))^{e_i}}}{\displaystyle{ \prod_{i=1}^N W_{E,P}(p_i(k))^{e_i} }} = (\phi(P)^{-1})^{t(k)}
\]
(Perhaps it is easier to demonstrate this concept by example:  suppose that $1=e_1=-e_2$, $p_1(k)=k+1$, and $p_2(k) = k$, so that $t(k) = 2k+1$, and obtain equation \eqref{eqn: 2perfper} (note the product \eqref{eqn: prod1} appears on the right side) from combining equation \eqref{eqn: perfper} for $n=p_1(k)=k$ and $n=p_2(k)=k+1$ with multiplicities given by $e_1=1$ and $e_2=-1$.)

The left hand side $A$ includes the known product \eqref{eqn: prod1} as well as terms of the form $\phi([p_i(k)]P)$, while $B = \phi(P)^{-1}$.  The $N$ points $[p_i(k)]P$ can each be calculated from knowledge of $P$ and $Q=[k]P$ without knowledge of $k$ in $O(D \log D)$ curve operations.  Then the various $\phi$ terms can be computed in time time $O((\log q)^3)$ by Theorem \ref{thm: calcphi}.  Thus we have computed $A$ and $B$.

Solving the discrete logarithm $A = B^{t(k)}$ for $t(k)$ can be done sub-exponentially by index calculus methods.  Solving for $k$ from $t(k)$ is sub-exponential \cite[\S 7.1-2]{BacSha}.\end{proof}

It is evident that the most costly step is the index calculus step, which in many cases has run time $r(q) = \exp( c (\log q)^{1/3} (\log \log q)^{2/3} )$ \cite[p.306]{CraPom}. 



%

\section{ECDLP and Quadratic Residues}
\label{sec: seqres1}

We will show that determining only one bit of information -- the residuosity -- about a term $W_{E,P}(k)$ may  suffice to solve the ECDLP.  First, we observe a hypothetical method of attack for ECDLP.

\begin{proposition}
\label{prop: parity}
Let $P$ be a point of odd order relatively prime to $q-1$.  Given an oracle which can determine the parity of the minimal multiplier of any non-zero point $Q$ in $\left< P \right>$ in time $O(T(q))$, the elliptic curve discrete logarithm for any such $Q$ can be determined in time $O(T(q) \log q +(\log q)^2)$.
\end{proposition}

\begin{proof}
Suppose that $k$ is the minimal multiplier of $Q$ with respect to $P$.  The basic algorithm is:
\begin{enumerate}
\item If $Q = P$, stop.
\item Call the oracle to determine the parity of $k$.  If $k$ is even, find $Q'$ such that $[2]Q' = Q$.  If $k$ is odd, find $Q'$ such that $[2]Q' = Q-P$.   
\item Set $Q = Q'$ and return to step $1$.
\end{enumerate}
In Step $2$, since the cyclic group $\left<P\right>$ has odd order, there is a unique $Q'$.  It can be found in $O(\log q)$ time (see \cite{FonHanLopMen} for methods).  Furthermore, $Q' = [k']P$ where
\[
k' = \left\{ \begin{array}{ll}
k/2 & k \mbox{ even} \\
(k-1)/2 & k \mbox{ odd}
\end{array} \right. .
\]
Then $k'$ is the minimal multiplier for $Q'$ with respect to $P$.  At the end of this process, the value of the original $k$ can be deduced from the sequence of steps taken.  For each even step, record a `0', and for each odd step a `1', writing from right to left, and adding a final `1': this will be the binary representation of $k$.  The number of steps is $\log_2 k = O(\log q)$.
\end{proof}

\begin{proposition}
\label{prop: qr}
Fix an elliptic curve $E$ defined over $\Fq$ of characteristic not equal to two, and $P \in E(\Fq)$ of order greater than three and relatively prime to $q-1$.  Suppose that $\phi(P)$ is a quadratic non-residue.  Then, with knowledge of the quadratic residuosity of any product of the form
\begin{equation}
\label{eqn: prod}
\prod_{i=1}^N W_{E,P}(p_i(k))^{e_i}, 
\end{equation}
where the $e_i \in \Z$, and $p_i(x) \in \Z[x]$ of degree at most $D$, and
$t(x) = \sum_{i=1}^N e_i p_i(x)^2$
is not constant as a function $\Z/2\Z \rightarrow \Z / 2\Z$, the parity of $k$ can be determined in time $O( N(D (\log D)(\log q)^2 + (\log q)^3) )$.
\end{proposition}

\begin{proof}
By Theorem \ref{thm: perfper}, the value $t(k)$ satisfies an equation in $\F_q^*$ of the form $A = B^{t(k)}$ (exactly as in the proof of Proposition \ref{prop: prod1}).  The quadratic residuosity of $A$ can be calculated in time $O(N ( D (\log D) (\log q)^2 + (\log q)^3) )$ as in the proof of Proposition \ref{prop: prod1}.  Now, $B = \phi(P)$ is a quadratic non-residue.  The parity of $t(k)$ can be calculated from these values in constant time (i.e. consider the question in $K^*$ modulo $(K^*)^2$).  The parity of $k$ is determined by checking the parity of $t(0)$ and $t(1)$.  This final step takes time $O(D)$.
\end{proof}

\begin{corollary}
\label{cor: qr}
Let $E$ be an elliptic curve over a field of characteristic not equal to two.  Let $P$ be a point of odd order such that $\phi(P)$ is a quadratic non-residue, and let $k$ be the minimal multiplier of a multiple $Q$ of $P$.  Given $P, Q$ and an oracle which can determine the quadratic residuosity of $W_{E,P}(k)$ in time $O(T(q))$, the elliptic curve discrete logarithm for any such $Q$ can be determined in time $O((\log q)(T(q)+(\log q)^3) )$.
\end{corollary}

\begin{proof}
This follows from Proposition \ref{prop: qr} with $N=1, e_1 = 1, p_1(x) = x$ and Proposition \ref{prop: parity}.
\end{proof}

\noindent
A few remarks are in order.
\begin{enumerate}
\item If $\phi(P)$ is a quadratic residue, one solution to this obstacle is to replace the initial problem of $Q = [k]P$ with the equivalent problem of $[n]Q = [k]([n]P)$ for any $n$ such that $\phi([n]P)$ is a quadratic non-residue.  The sequence $\widetilde W_{E,P}(n)$ can be calculated term-by-term until such an $n$ is found.  The existence of such an $n$ is guaranteed when $-1$ is a quadratic non-residue in $\Fq$, in which case $\phi([m-1]P) = -\phi(P)$ suffices.  Other cases are less clear.
\item The condition that the order of $P$ is relatively prime to the even quantity $q-1$ is required in several ways.  First, for the very definition of $\phi$ (Theorem \ref{thm: perfper}).  Furthermore, if the order $m$ of the group $\left< P \right>$ is even, multiplication by $2$ is not an automorphism, and so there is no unique `half' of a point (this is the same difficulty that prevents this sort of parity attack on an $\Fq^*$ discrete log).  However, if $m | (q-1)$ is odd, then $k$ satisfies a discrete logarithm equation of the form $A=B^k$ in the group $K^* / (K^*)^m$, which has an odd number of elements.  Therefore, this does not determine the parity of $k$.
\item Similarly, if $q-1$ is odd (i.e. $\Fq$ has characteristic $2$), then $A=B^k$ does not carry information about the parity of $k$.
\end{enumerate}
%

\section{The EDS Residue Problem}
\label{sec: seqres2}

In light of the preceeding section, it is natural to define the problem of EDS Residue (Problem \ref{prob: seqres}).  In Section \ref{sec: equiv} we will show that it is equivalent to the elliptic curve discrete logarithm in sub-exponential time.  How might one determine the quadratic residuosity of $W_{E,P}(k)$?  Our first observation is that knowledge of the residuosity of one term $W_{E,P}(k)$ would determine the residuosity of the next term.

\begin{proposition}
Suppose $Q$ is a known element of $\left< P \right>$, but that its minimal multiplier $k$ is unknown.  The quadratic residuosity of $W_{E,P}(k+1)/W_{E,P}(k)$ can be calculated in $O((\log q)^3)$ time.
\end{proposition}

\begin{proof}
From \eqref{eqn: perfper} with $n=k$ and $n=k+1$, we have
\[
\frac{\phi(Q)}{\phi(Q+P)} = \phi(P)^{2k+1}\left(\frac{W_{E,P}(k+1)}{W_{E,P}(k)}\right).
\]
The calculation of the terms $\phi(P), \phi(Q),$ and $\phi(P+Q)$ each take $O((\log q)^3)$ time.
\end{proof}

Therefore, based on knowledge of $Q$ but not $k$, the sequence $$ S(n) = \left( \frac{W_{E,P}(n)}{q} \right)\left( \frac{W_{E,P}(k)}{q} \right) $$for $n=k, \ldots, k+N$ may be calculated in $O(N\log q)$ time.  Then the sequence $$\left( \frac{W_{E,P}(n)}{q} \right)$$ is either $S(n)$ or $-S(n)$.  To determine which is to determine the quadratic residuosity of $W_{E,P}(k)$.

Therefore, if some bias, or some pattern, for quadratic residues of the elliptic divisibility sequence $W_{E,P}(n)$ were known, then the correct choice of the two sequences above could be determined.  However, 
as yet we have no evidence to suggest that
the ratio of quadratic residues among the terms is not $1/2$ in general.  

\section{ECDLP through EDS Discrete Log in the case of Perfect Periodicity}
\label{sec: edsdlp}

Problem \ref{prob: edsdlp} (EDS Discrete Log) is less unusual in flavour than the other problems considered here:  general discrete logarithm attacks will apply.  Recall the proof of Theorem \ref{thm: shipsey}, in which \emph{blocks centred at $k$} are defined -- denote this as $B(k)$.  From $B(k)$, the recurrence relation can be used to calculate $B(2k)$ or $B(2k+1)$.  In fact, Shipsey goes further, and shows how two blocks $B(k), B(k')$ can be added to obtain a block $B(k+k')$ in a similarly efficient manner (see \cite[p. 23]{Shi}).  This means that the sequence of blocks $B(n)$ is a sequence along which we can move easily by addition and $\Z$-multiplication.  Therefore, algorithms such as Baby-Step-Giant-Step and Pollard's $\rho$ can be applied to this problem.

\section{Equivalence of Hard Problems}
\label{sec: equiv}

\begin{proof}[Proof of Theorem \ref{thm: equiv}]
{\bf $(3) \implies (1)$:}  Corollary \ref{cor: qr}.  
{\bf $(1) \implies (2)$:}  If $k$ is known, we can assume $0 < k \leq \operatorname{ord}(P)$, and then $W_{E,P}(k)$ can be calculated in $O((\log k) (\log q)^2)=O((\log q)^3)$ time.  
{\bf $(2) \implies (3)$:}  Residuosity of a value in $\Fq^*$ can be determined in sub-exponential time (see \cite{ItoTsu} for algorithms).
{\bf $(1) \implies (4)$:}  Theorem \ref{thm: calcfromphi}.
{\bf $(4) \implies (1)$:}  Theorem \ref{thm: calcphi} allows calculation of $\phi([k]P)$, $\phi([k+1]P)$, and $\phi([k+2]P)$ in sub-exponential time.
\end{proof}

\bibliographystyle{splncs}
\bibliography{ellnet}

\def\cprime{$'$}
\begin{thebibliography}{10}

\bibitem{Shi}
Shipsey, R.:
\newblock Elliptic Divibility Sequences.
\newblock PhD thesis, Goldsmiths, University of London (2001)

\bibitem{War}
Ward, M.:
\newblock Memoir on elliptic divisibility sequences.
\newblock Amer. J. Math. \textbf{70} (1948)  31--74

\bibitem{Swa}
Swart, C.:
\newblock Elliptic curves and related sequences.
\newblock PhD thesis, Royal Holloway and Bedford New College, University of
  London (2003)

\bibitem{Aya}
Ayad, M.:
\newblock P\'eriodicit\'e (mod {$q$}) des suites elliptiques et points
  {$S$}-entiers sur les courbes elliptiques.
\newblock Ann. Inst. Fourier (Grenoble) \textbf{43}(3) (1993)  585--618

\bibitem{Sil5}
Silverman, J.H.:
\newblock Common divisors of elliptic divisibility sequences over function
  fields.
\newblock Manuscripta Math. \textbf{114}(4) (2004)  431--446

\bibitem{Sil4}
Silverman, J.H.:
\newblock {$p$}-adic properties of division polynomials and elliptic
  divisibility sequences.
\newblock Math. Ann. \textbf{332}(2) (2005)  443--471 (Addendum 473--474)

\bibitem{EveMclWar}
Everest, G., Mclaren, G., Ward, T.:
\newblock Primitive divisors of elliptic divisibility sequences.
\newblock J. Number Theory \textbf{118}(1) (2006)  71--89

\bibitem{Sta4}
Stange, K.E.:
\newblock Elliptic nets and elliptic curves.
\newblock \url{http://arxiv.org/abs/0710.1316v1}, submitted (2007)

\bibitem{Sta1}
Stange, K.E.:
\newblock Elliptic nets and elliptic curves.
\newblock PhD thesis, Brown University (May 2008)

\bibitem{EvePooShpWar}
Everest, G., Poorten, A.v.d., Shparlinski, I., Ward, T.:
\newblock Elliptic Divisibility Sequences.
\newblock In: Recurrence Sequences. American Mathematical Society, Providence
  (2003)  163--175

\bibitem{Sta3}
Stange, K.E.:
\newblock The {T}ate pairing via elliptic nets.
\newblock In: Pairing-Based Cryptography - PAIRING 2007. Volume 4575 of Lecture
  Notes in Comput. Sci.
\newblock Springer, Berlin (2007)  329--348

\bibitem{GosOrmSch}
Gosper, R.~W., O.H., Schroeppel, R.:
\newblock Using somos sequences for cryptography

\bibitem{Sil1}
Silverman, J.H.:
\newblock The arithmetic of elliptic curves. Volume 106 of Graduate Texts in
  Mathematics.
\newblock Springer-Verlag, New York (1992) Corrected reprint of the 1986
  original.

\bibitem{Sta5}
Stange, K.E.:
\newblock Elliptic nets, generalised {J}acobians and bi-extensions.
\newblock In preparation

\bibitem{FreLan}
Frey, G., Lange, T.:
\newblock Background on curves and {J}acobians.
\newblock In: Handbook of elliptic and hyperelliptic curve cryptography.
  Discrete Math. Appl. (Boca Raton).
\newblock Chapman \& Hall/CRC, Boca Raton, FL (2006)  45--85

\bibitem{BacSha}
Bach, E., Shallit, J.:
\newblock Algorithmic number theory. {V}ol. 1.
\newblock Foundations of Computing Series. MIT Press, Cambridge, MA (1996)
  Efficient algorithms.

\bibitem{DuqFre2}
Duquesne, S., Frey, G.:
\newblock Background on pairings.
\newblock In: Handbook of elliptic and hyperelliptic curve cryptography.
  Discrete Math. Appl. (Boca Raton).
\newblock Chapman \& Hall/CRC, Boca Raton, FL (2006)  115--124

\bibitem{Gal}
Galbraith, S.D.:
\newblock Pairings.
\newblock In: Advances in elliptic curve cryptography. Volume 317 of London
  Math. Soc. Lecture Note Ser.
\newblock Cambridge Univ. Press, Cambridge (2005)  183--213

\bibitem{MenOkaVan}
Menezes, A.J., Okamoto, T., Vanstone, S.A.:
\newblock Reducing elliptic curve logarithms to logarithms in a finite field.
\newblock IEEE Trans. Inform. Theory \textbf{39}(5) (1993)  1639--1646

\bibitem{FreRuc}
Frey, G., R{\"u}ck, H.G.:
\newblock A remark concerning {$m$}-divisibility and the discrete logarithm in
  the divisor class group of curves.
\newblock Math. Comp. \textbf{62}(206) (1994)  865--874

\bibitem{CraPom}
Crandall, R., Pomerance, C.:
\newblock Prime numbers.
\newblock Springer-Verlag, New York (2001) A computational perspective.

\bibitem{FonHanLopMen}
Kenny~Fong, Darrel~Hankerson, J.L., Menezes, A.:
\newblock Field inversion and point halving revisited.
\newblock Technical Report, CORR 2003-18, Department of Combinatorics and
  Optimization, University of Waterloo, Canada (2003)

\bibitem{ItoTsu}
Itoh, T., Tsujii, S.:
\newblock An efficient algorithm for deciding quadratic residuosity in finite
  fields {${\rm GF}(p\sp m)$}.
\newblock Inform. Process. Lett. \textbf{30}(3) (1989)  111--114

\end{thebibliography}

\end{document}